\documentclass[a4paper]{article}
\usepackage{amssymb}
\usepackage{graphicx,color}
\usepackage{amsthm}
\usepackage{amsmath}
\usepackage{enumerate}
\usepackage{tikz}

\usepackage{float}
\usepackage{subfig}
\usepackage{hyperref}
\usepackage{multirow}
\usepackage{rotating}
\usepackage{bbm}

\theoremstyle{plain}
\newtheorem{theorem}{Theorem}

\newtheorem{corollary}[theorem]{Corollary}
\newtheorem{lemma}[theorem]{Lemma}
\newtheorem{conjecture}[theorem]{Conjecture}

\theoremstyle{definition}

\newtheorem{definition}[theorem]{Definition}

\def\ra{\rightarrow}

\def\ba{\begin{array}}
\def\ea{\end{array}}
\def\bi{\begin{itemize}}
\def\ei{\end{itemize}}
\def\mP{\mathbb{P}}
\def\mR{\mathbb{R}}
\def\mZ{\mathbb{Z}}

\def\m1{\mathbbm{1}}

\def\cN{\mathcal{N}}
\def\cL{\mathcal{L}}

\def\eps{\varepsilon}

\let\oldhat\hat

\renewcommand{\hat}[1]{\oldhat{\mathbf{#1}}}

\providecommand{\com[2]}{{\color{blue} \begin{tt}[#1: #2]\end{tt}}}

\begin{document}
\title{Analysis of a non-reversible Markov chain speedup by a single edge}
\author{Bal\'azs Gerencs\'er\thanks{B. Gerencs\'er is with the Alfr\'ed R\'enyi Institute
    of Mathematics, Hungarian Academy of Sciences and
the ELTE E\"otv\"os Lor\'and University, Department of Probability Theory and Statistics,
  {\tt\small gerencser.balazs@renyi.mta.hu}. His work is supported by NKFIH (National Research, Development and Innovation Office) grant PD 121107 and KH 126505.}%
}
\date{\today}

\maketitle

\begin{abstract}
  We present a Markov chain example where non-reversibility and an
  added edge jointly improve mixing time: when a random edge is added
  to a cycle of $n$ vertices and a Markov chain with a drift is
  introduced, we get a mixing time of $O(n^{3/2})$ with probability
  bounded away from 0. If only one of the two modifications were
  performed, the mixing time would stay $\Omega(n^2)$.
\end{abstract}



\section{Introduction}

The fundamentals of Markov chain theory is well established, but it is
still in constant development due to diverse motivations from
applications and inspiring sparks from novel discoveries
\cite{aldous-fill-2014}, \cite{levin2017markov}.
Understanding mixing gives an insight on the macroscopic behavior of the
dynamics of the chain, moreover it is also a crucial factor determining the
efficiency of applications built using the chain.
The Markov chain Monte Carlo approach is one of the popular scheme to
translate mixing of Markov chains into powerful methods for sampling
or numerical integration \cite{diaconis:montecarlo_broadsurvey2009}.

Simple examples realizing new phenomena, either awaited or surprising,
help the community get a deeper understanding on what and how is
possible. The aim of the current paper is to present and discuss such
an
example.

The starting point is the cycle with $n$ vertices, see Figure
\ref{fig:korall} (a), where we consider the 
Markov chains with uniform stationary distribution. For reversible Markov chains (meaning that the stationary
frequency of transition along every edge is the same in the two
directions) standard results show that the mixing time is $\Omega(n^2)$.

\begin{figure}[h]
  \centering
  \hspace{0.05\textwidth}
  \subfloat[ ]{
    \includegraphics[width=0.23\textwidth]{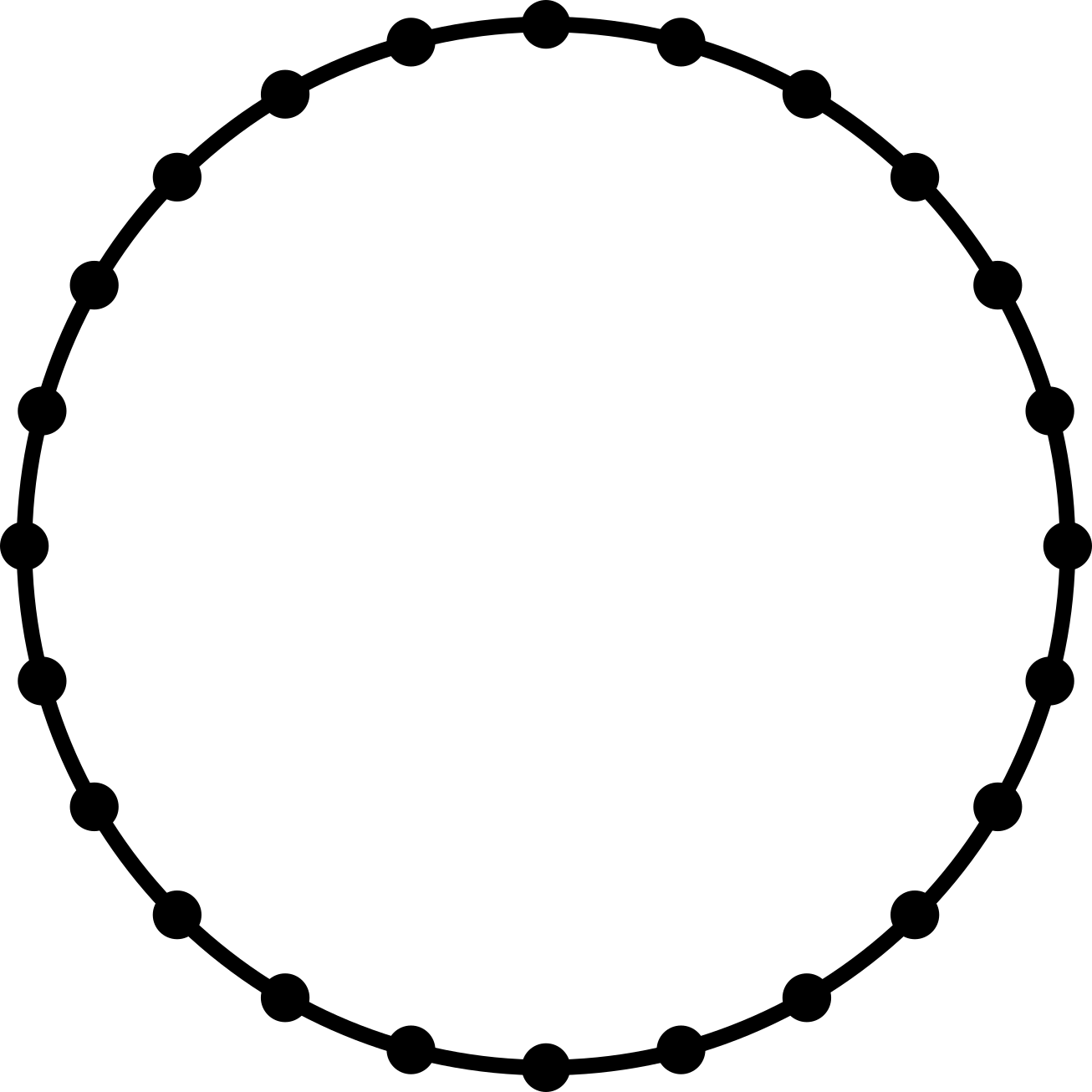}
  }
  \hspace{0.05\textwidth}
  \subfloat[ ]{
    \includegraphics[width=0.23\textwidth]{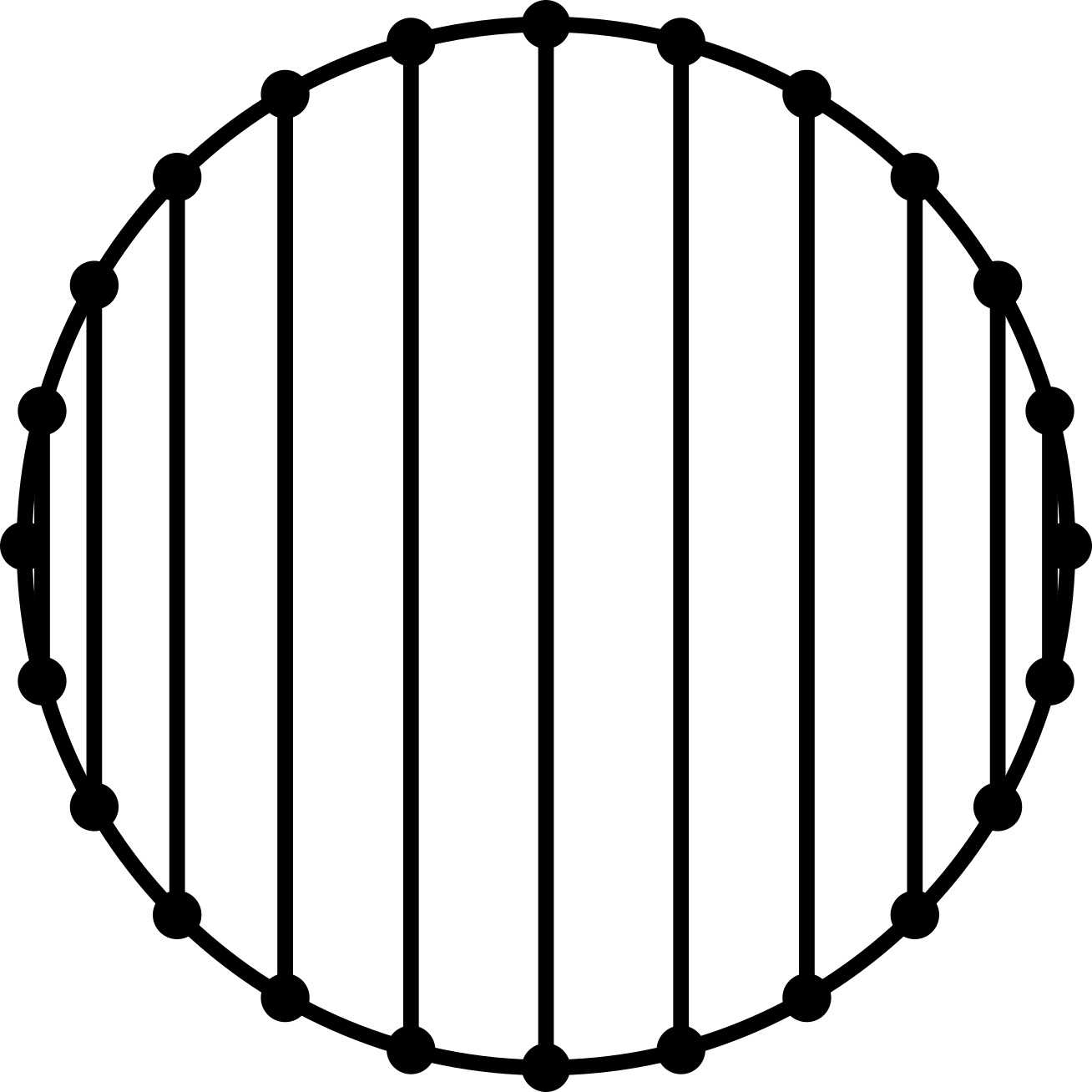}
  }
  \hspace{0.05\textwidth}
  \subfloat[ ]{
    \includegraphics[width=0.23\textwidth]{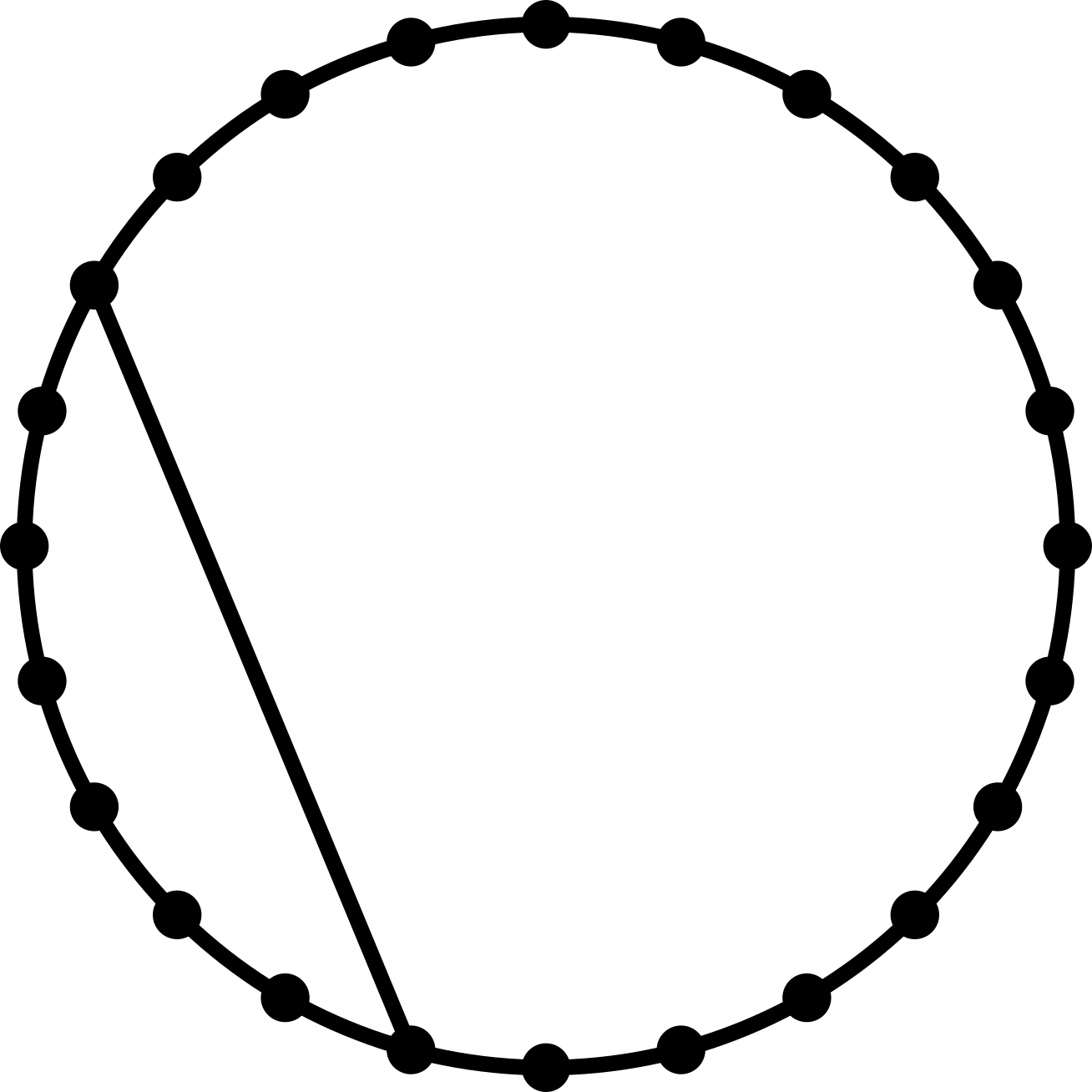}
  }
  \hspace{0.05\textwidth}
  \caption{Cycle graph with either (a) no extra edges; (b) all parallel edges
    between opposing vertices; (c) a single random extra edge.}
  \label{fig:korall}
\end{figure}

Relaxing the reversibility
condition does not help in this simple case. An important observation is that we only get
a single extra degree of freedom, a possible \emph{drift}: we may increase all
clockwise and decrease all counter-clockwise transition probabilities
by the same amount
departing from a reversible transition structure. It is a surprisingly non-trivial
result of the author \cite{gb:ringmixing2011} that the mixing time is still $\Omega(n^2)$.

Still striving for faster mixing, we may analyze a graph with additional
edges. It was a key revelation of Diaconis, Holmes and Neal \cite{diaconis:nonrev2000} that
when edges connecting opposing vertices are available, as in Figure
\ref{fig:korall} (b) \emph{and} a strong drift is used,
the mixing time drops to $O(n)$ (this is a slight reinterpretation of
their context).

What happens if we use fewer extra edges? In this paper we want to
understand the other extreme, when the number of added edges is 1. We
choose a single extra edge randomly, see
Figure \ref{fig:korall} (c). For any reversible chain, the mixing time
is again $\Omega(n^2)$ as there is still a path of length of order $n$
without any other edge.

However, if again combined with a drift along the cycle, we may
get a considerable speedup, with the mixing time decreasing to
$O(n^{3/2})$. Let us now proceed to the next section with the
necessary formal definitions and precise statement of the theorem.

\section{Preliminaries and main result}

Formally the Markov chains of interest can be built as follows.

\begin{definition}
  \label{def:model}
  Consider a cycle graph of $n$ vertices on $V=\{1,2,\ldots,n\}$. Choose a uniform random integer $k$
  from $[2,n-2]$.
  We will name
  vertices $k,n$ as \emph{hubs} and add the new edge $(k,n)$. For
  convenience, on the arc $1,2,\ldots,k$ we introduce
  the notation $a_1,a_2,\ldots, a_k$ for the vertices and name the arc $A$ while on the
  arc $k+1,k+2,\ldots,n$ we will use $b_1,b_2,\ldots,b_{n-k}$ and $B$.
  Define the Markov chain transition probabilities as follows. Set
  \begin{itemize}
  \item $P(a_i,a_{i+1})=P(b_j,b_{j+1})=1/2$ for all $1\le i \le k-1,~
    1\le j \le n-k-1$,
  \item $P(a_k,b_1)=P(b_{n-k},a_1)=1/2$,
  \item $P(a_i,a_i)=P(b_j,b_j)=1/2$ for all $1\le i \le k-1,~
    1\le j \le n-k-1$,
  \item $P(a_k,a_k)=P(a_k,b_{n-k})=P(b_{n-k},a_k)=P(b_{n-k},b_{n-k})=1/4$.
  \end{itemize}
\end{definition}
Set all other entries to 0.

It is easy to verify that this transition kernel is doubly stochastic,
therefore it is a valid transition kernel with the uniform
distribution as the stationary distribution (aperiodicity,
irreducibility ensures uniqueness).

We denote the Markov chain of interest on the cycle by $X(t)$, that
is, choosing $X(0)$ according to some preference and then using the
transition probabilities $P$ defined above.

We are going to compare probability distributions with their \emph{total
  variation distance}. For any $\mu,\sigma$ probability distributions
this is defined as
$$
\|\mu-\sigma\|_{\rm TV} := \sup_{S\subseteq V}
|\mu(S)-\sigma(S)| = \frac{1}{2}\sum_{x\in V}
|\mu(x)-\sigma(x)|.
$$
Keeping in mind that currently the stationary distribution is uniform,
which we denote by $\frac{\m1}{n}$, we define the maximal
distance and the \emph{mixing time} as
\begin{align}
d_X(t) &:=\sup_{X(0)\in V}\left\|\cL(X(t))-\frac{\m1}{n}\right\|_{\rm TV}, \label{eq:ddef}\\
  t_{\rm mix}(X,\eps)&:=\min\left\{t: d_X(t) \le \eps\right\}. \label{eq:tmixdef}
\end{align}
We now have all the ingredients to state our main results: 

\begin{theorem}
  \label{thm:main_upper_bound}
  There exist $\gamma, \gamma'>0$ constants such that for any $1/2>\eps>0$ the
  following holds. For the (randomized) Markov chain of Definition
  \ref{def:model}, for $n$ large enough, with probability at least
  $\gamma$ we have
  $$
  t_{\rm mix}(\eps) \le \gamma' n^{3/2} \log\frac{1}{\eps}.
  $$
\end{theorem}

\begin{theorem}
  \label{thm:main_lower_bound}
  There exist $\gamma^*, \eps^*>0$ constants such that for the (randomized) Markov chain of Definition
  \ref{def:model}, for $n$ large enough we (deterministically) have
  $$
  t_{\rm mix}(\eps^*) \ge \gamma^* n^{3/2}.
  $$
\end{theorem}

During the statements and proofs we will have various constants
appearing. As a general rule, we use $\gamma_i$ for the statements
that depend on each other using increasing indices to express dependence, while
$\gamma,\gamma',\gamma^*$ and $\eps^*$ of Theorem
\ref{thm:main_upper_bound} and \ref{thm:main_lower_bound} might depend on
all of them. We will carry through a time scaling constant $\rho$
which we can take to be 1 for most of the paper but we will need this
extra flexibility for the lower bound.
We use various $\beta_j$ during the proofs
and only for the scope of the proof, they might be reused later.
Currently our Markov chain is based on the cycle so we will use the metric
on it, therefore $|i-j|$ will often be used as a simplification for
$\min(|i-j|, n-|i-j|)$ when appropriate.

The rest of the paper is organized as follows.
First we analyze the path taken
after a certain number of moves and record it on a square grid. We give estimates on
reaching certain points of interests on this grid. This is worked out
in Section \ref{sec:grid}. Afterwards, we investigate the connection
of the grid to our original graph in Section \ref{sec:gridmap}. In
Section \ref{sec:diffusion} we switch from possible tracks to the actual
Markov chain and examine the appearing diffusive randomness. In
Section \ref{sec:mixproof} we join the elements to get a complete proof of
Theorem \ref{thm:main_upper_bound}. We adjust this viewpoint to obtain
the lower bound of Theorem \ref{thm:main_lower_bound} in Section
\ref{sec:lowerbound}. Finally, we conclude in Section \ref{sec:discussion}.

\section{Grid representation of the tracks}
\label{sec:grid}

As a first step we want to understand the track of the Markov chain it
traverses without the time taken to walk through it. That is, we
disregard loop steps and we also disregard (possibly repeated) hopping
between hubs, in this sense we will often treat the two hubs
together. This means that for our purpose the track of the Markov chain is composed of traversing an
arc, then choosing the next when reaching the pair of hubs, then traversing
another one and so on.

In order to represent this, consider the non-negative quadrant of the square integer lattice together
with a direction, that is $\mZ_+^2\times H$, where $H = \{A,B\}$. A position
position $(x,y,h)$ will represent that the Markov chain has taken the $A$ arc
$x$ times, the $B$ arc $y$ times and it arrived to the position on a
horizontal ($A$) or vertical ($B$) segment. We may also translate intermediate
points of the grid lines of the non-negative quadrant as being within a certain arc (after
completely traversing some previously), here the direction $h$ is
predetermined by the direction of the segment where the point lies. Note that at
grid points of $\mZ_+^2$ the arcs would overlap, representing both
$a_k$ and $b_{n-k}$, this is the reason for the extra property of direction.
At the initial point there are no traversed arcs yet, $X(0)$ therefore corresponds to a point $(-\lambda, 0, A)$ or
$(0,-\lambda, B)$ for some $\lambda\in [0,1)$. This structure can be seen in Figure \ref{fig:gridalap}.

\begin{figure}[h]
  \centering
  \includegraphics[width=0.4\textwidth]{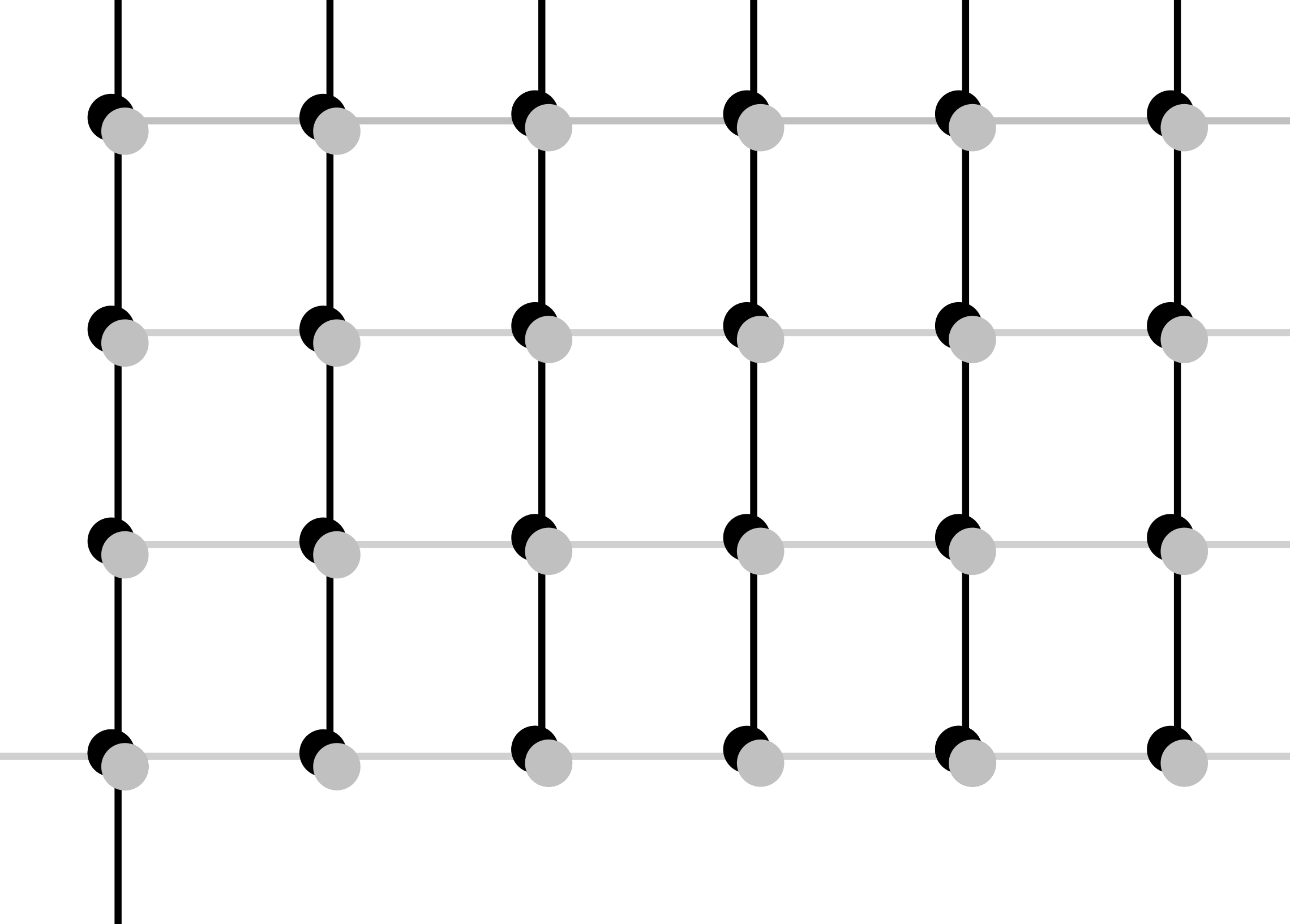}
  \caption{Illustration of $\mZ_+^2\times H$. Points of form
    $(\cdot,\cdot,A)$ are represented by gray dots
    while $(\cdot,\cdot,B)$ points are represented by black dots. The extension to grid lines is also shown,
    with gray and black lines corresponding to points having $A$ and
    $B$ last coordinate, respectively.}
  \label{fig:gridalap}
\end{figure}

We are interested at the track at some $L=\rho n^{3/2} + O(n)\in\mZ$
distance on the graph. This is represented by $\approx
\rho n^{1/2}$ distance in the current grid where each segment represents an
arc of the cycle. Formally, we need the slanted line of the
following intersection points, see also Figure \ref{fig:gridR}.
\begin{equation}
  \label{eq:Rdef}
  R = \{(x,y,h)\in \mR_+^2\times H ~|~ (x\in\mZ_+,h=B) \textrm{ or } (y\in\mZ_+,h=A), kx+(n-k)y=L\},
\end{equation}

\begin{figure}[h]
  \centering
  \includegraphics[width=0.4\textwidth]{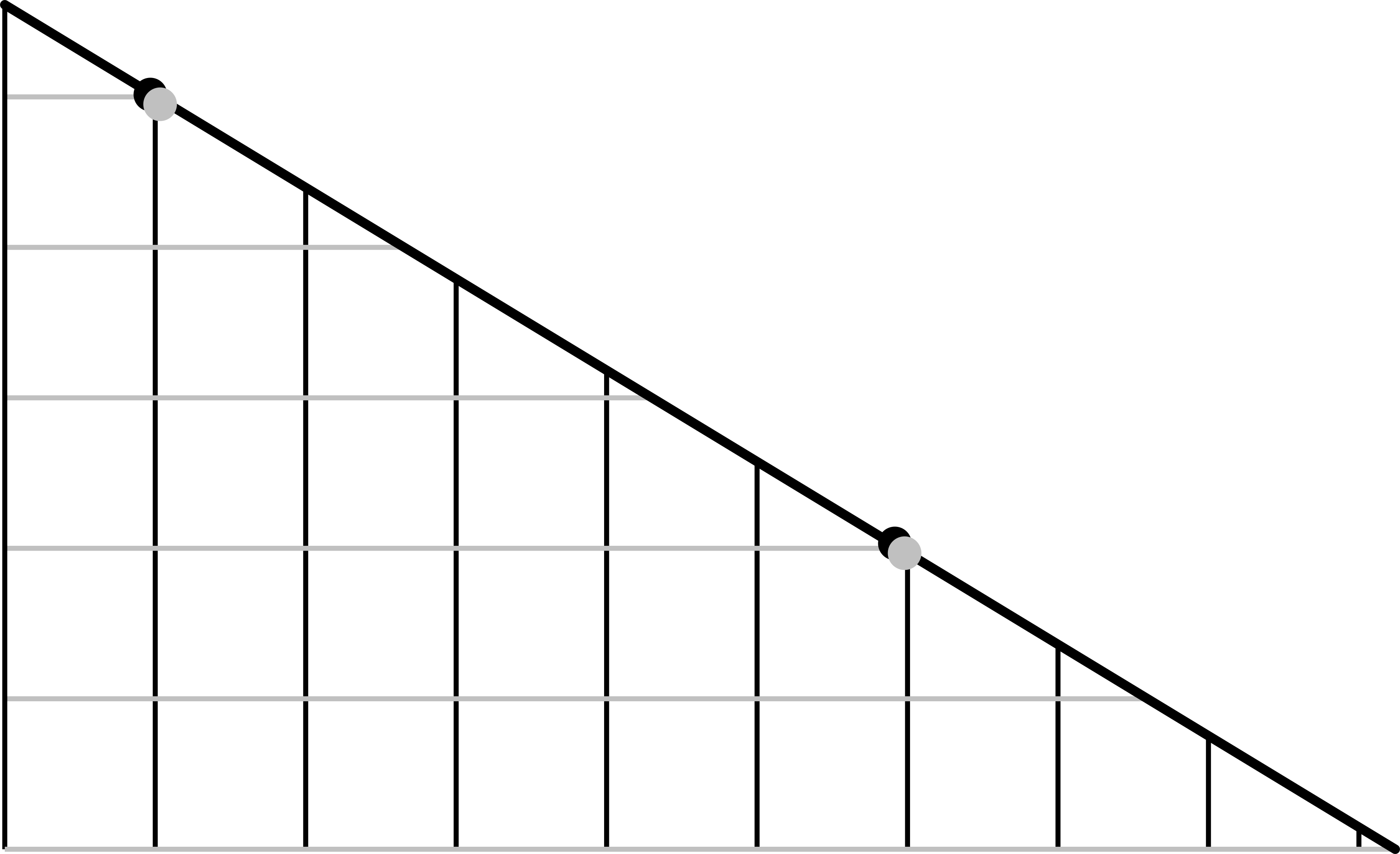}
  \caption{Illustration of $R$. An intersection of the slanted line with
    a grid line represents a single point in $R$ while passing through
    a grid point corresponds to two points in $R$.}
  \label{fig:gridR}
\end{figure}
Note that the constraints on the coordinates of the points of $R$
are sufficient so that
they do properly
map to vertices of the graph. Let us denote by $g:R\ra V$ this
function that recovers the vertex of the graph given a point in $R$.
For $r\in R$ let us denote by $E_r$ the
event that the Markov chain reaches $r$. Observe that these present a
mutually exclusive partitioning of the probability space (except a
null-set where the movement of the Markov-chain is bounded).

\subsection{Exact expression for $E_r$ probabilities}

The main goal of this section is to get usable bounds for
$\mP(E_r)$. Let us examine how the Markov chain may reach $r$.
The transition probabilities at the hubs imply that after traversing a
horizontal unit segment there is $3/4$ probability to choose a
vertical one and $1/4$ to take another horizontal one, and vice
versa (and no backtracking can happen). For a fixed $r=(x,y,h)$ let us define the preceding grid point
on $\mZ_+^2$ in the $h$ direction $r'=(x',y')$:
$$
r' = (x',y') = \begin{cases}
  \lceil x \rceil -1, y &\text{ if } h=A,\\
  x, \lceil y \rceil -1 &\text{ if } h=B.
\end{cases}
$$
We will simply use $x', y'$ by themselves if $r$ is clear from
the context.
\begin{lemma}
  For $r\in R$ we have the following probability bound for reaching it:
  \begin{equation}
    \label{eq:binom-global}
    \mP(E_r) \ge \frac 1 4 \left[Binom(x',3/4) *
      Binom(y',1/4)\right](y') \ge \frac{1}{9}\mP(E_r),
  \end{equation}
  where the $Binom$ stand for Binomial distributions and $*$ indicate
  their convolution.
\end{lemma}
\begin{proof}
For the moment,
let us assume that $r_H=B$, $X(0)_H=A$, and $x',y'>0$. The calculations
for all other cases work similarly, we are going to discuss the
differences afterwards.

We need to gather all possible routes to $r$ together with their
probabilities. Assume that there were $2i+1$ changes of direction (it
is certainly odd as a horizontal start ended on a vertical
segment). This already determines that the probability of such a path
is
$$
\left(\frac 3 4\right)^{2i+1} \left(\frac 1 4\right)^{x'+y'-2i}.
$$
To achieve $2i+1$ changes of directions, we need to split the total $x'$
horizontal steps into $i+1$ parts, noting that a change may be valid
after 0 steps due to the horizontal initialization. Consequently there
are $x' \choose i$ ways to do this. Similarly, there are $y' \choose
i$ possibilities to arrange the vertical steps. Altogether, we get
$$
\mP(E_r) = \sum_{i=0}^\infty {x' \choose i} {y' \choose i} \left(\frac 3 4\right)^{2i+1} \left(\frac 1 4\right)^{x'+y'-2i}.
$$
(The summation to $\infty$ is just for notational convenience, clearly
there are only finitely many non-zero terms.)
We may reformulate this as follows:
\begin{equation}
  \label{eq:binom-decomposition}
  \begin{aligned}
    \mP(E_r) &= \sum_{i=0}^\infty {x' \choose i} \left(\frac 3 4\right)^{i+1} \left(\frac 1 4\right)^{x'-i}
    {y' \choose y'-i} \left(\frac 1 4\right)^{y'-i} \left(\frac 3 4\right)^{i}\\
    &= \frac 3 4 \left[Binom(x',3/4) * Binom(y',1/4)\right](y'). 
  \end{aligned}
\end{equation}
For all the scenarios in terms of the orientation of the first and last segment,
we can perform the same computation. Collecting all the similar
expressions together we get:
\begin{align}
  X(0)_H=A, r_H=B \quad\Rightarrow\quad \mP(E_r) &= \frac 3 4 \left[Binom(x',3/4) * Binom(y',1/4)\right](y').\label{eq:binom-AB}\\
  X(0)_H=A, r_H=A \quad\Rightarrow\quad \mP(E_r) &= \frac 3 4 \left[Binom(x'+1,3/4) * Binom(y'-1,1/4)\right](y'),\label{eq:binom-AA}\\
  X(0)_H=B,r_H=B \quad\Rightarrow\quad \mP(E_r) &= \frac 3 4 \left[Binom(x'-1,3/4) * Binom(y'+1,1/4)\right](y'),\\
  X(0)_H=B,r_H=A \quad\Rightarrow\quad \mP(E_r) &= \frac 3 4 \left[Binom(x',3/4) * Binom(y',1/4)\right](y').
\end{align}
For our further work it will be convenient to avoid case
splitting. Observe that each distribution is a sum of $x'+y'$ bits with
very minor difference. This means that, for instance comparing
\eqref{eq:binom-AB} and \eqref{eq:binom-AA}, we have for any $s \in
[0,1,\ldots, x'+y']$
$$
\frac 1 3 \le
\frac{\left[Binom(x'+1,3/4) * Binom(y'-1,1/4)\right](s)}
{\left[Binom(x',3/4) * Binom(y',1/4)\right](s)}
\le 3,
$$
as the two distributions can be perfectly coupled except for a single bit of
probability $3/4$ being coupled to a bit of probability
$1/4$. Using the same comparison for the other cases with the reference
\eqref{eq:binom-AB} 
and accepting this error margin we get the overall bounds
\begin{align*}
  \mP(E_r) &\ge \frac 1 4 \left[Binom(x',3/4) *  Binom(y',1/4)\right](y'),\\
    \frac 1 3 \mP(E_r) &\le \frac 3 4 \left[Binom(x',3/4) * Binom(y',1/4)\right](y'),
\end{align*}
which combine together to give the statement of the lemma.
This is now valid irrespective of directions. Even more, this final
inequality holds true when $x'=0$ or $y'=0$ (although we will not need
it). 
\end{proof}

In the next subsection we aim to bound this in a simpler form. We want to emphasize that
using the two Binomial distributions and evaluating their convolution
at one point provides the probability for a single $E_r$, different
distributions are needed for other elements of $R$.

\subsection{Simplified estimates of $E_r$ probabilities}

We will give estimates for points which are close to the
diagonal. Define
$$
R_0 = \{r \in R ~|~ |x'-y'| \le \sqrt{\rho}n^{1/4}\}.
$$
By the definition \eqref{eq:Rdef}
of $R$ we get
$$\min(x',y')n \le L \le \max(x',y')n,$$
which ensures $x',y' \in \left[\rho n^{1/2} - \sqrt{\rho}n^{1/4} + O(1), \rho n^{1/2} +
  \sqrt{\rho}n^{1/4} + O(1)\right]$ within this set.
\begin{lemma}
  There exists constant $\gamma_1>0$ such that for $n$ large enough and any point $r\in R_0$ we have
  \begin{equation}
    \label{eq:PEr-bound}
    \mP(E_r) \ge \frac{\gamma_1}{\sqrt{\rho}n^{1/4}}.
  \end{equation}
\end{lemma}

Clearly this is what one would expect from CLT asymptotics,
and such bounds are widely available for simple Binomial distributions. Here is one
possible way to confirm
the claim for our non-homogeneous case.
\begin{proof}
Let
$$
Q = Binom(x',3/4) * Binom(y',1/4),
$$
the distribution appearing in
\eqref{eq:binom-decomposition}. It can be viewed as the sum of $x'+y'$
independent indicators. We will approximate $Q$ with a Gaussian
variable in a quantitative way using the Berry-Esseen theorem.

For an indicator with probability $1/4$ the variance is $3/16$, the
absolute third moment after centralizing is $15/128$, we get the same
values for the indicator with probability $3/4$ due to
symmetry. Consequently we may consider the approximation
$$
Q \approx \cN\left(\frac{3x'+y'}{4},\frac{3x'+3y'}{16}\right).
$$
Denoting by $F_Q$ and $F_\cN$ the cumulative
distribution functions of $Q$ and the properly scaled Gaussian above,
the Berry-Esseen theorem (for not identically distributed variables)
ensures
\begin{equation}
  \label{eq:berryesseen}
  \sup_{\xi\in\mR} |F_{Q}(\xi) - F_{\cN}(\xi)|\le \beta_1 \left(\frac{3}{16}(x'+y')\right)^{-3/2} \frac{15}{128}(x'+y') = \frac{\beta_2}{(x'+y')^{1/2}},
\end{equation}
where $\beta_1>0$ is the global constant of the theorem.
Combining with
the necessary second and third moments of our sum of independent
variables added up, we get an explicit constant $\beta_2>0$ for the current
problem. Note that the Berry-Essen theorem is originally stated for
centered and normalized sums, but joint shifting and scaling does not
change the difference of the cumulative distribution functions. 

Let us now introduce the normalized distribution $\tilde{Q}$ by
$$
\tilde{Q}\left(\left(C-\frac{3x'+y'}{4}\right)\frac{4}{\sqrt{3x'+3y'}}\right) = Q(C)
$$
for any measurable set $C$, which is then approximated by the standard
Gaussian distribution $\Phi$ (but is still a discrete distribution). This definition implies that in
\eqref{eq:binom-decomposition} we need the value of
$$
\beta_3=\tilde{Q}\left(\left\{\alpha\right\}\right),\quad
  \text{where }\alpha = \frac{\sqrt{3}(y'-x')}{\sqrt{x'+y'}}.
$$
Observe that by the definition of $R_0$  we have $|\alpha|<\sqrt{3/2}+O(n^{-1/8})<3/2$, for $n$ large enough. Define the intervals
$$
I_-=\left[-2,-\frac 3 2\right],\quad
I_+=\left[\frac 3 2, 2\right].
$$
Recall that Binomial distributions are log-concave, so is their
convolution $Q$, and its affine modification $\tilde{Q}$. Consequently,
for any grid point the probability is at least all those that precede
it or at least all those that are after it. In particular, $\beta_3$ is
bounded below by all the (grid point) probabilities of $I_-$ or
$I_+$. Simplifying further, we can take the average probabilities on the
intervals, the lower will be a valid lower bound for $\beta_3$.

By the Berry-Essen estimate \eqref{eq:berryesseen} we have that
\begin{align*}
  q_-= \tilde{Q}(I_-) &= \Phi\left(\left[-2, -\frac 3 2\right]\right) + O(n^{-1/4}),\\
  q_+ = \tilde{Q}(I_+) &= \Phi\left(\left[\frac 3 2, 2\right]\right) + O(n^{-1/4}).
\end{align*}

To estimate the number of grid points in the two intervals, we refer back to the
unnormalized distribution $Q$ where we have to count the integers in
the corresponding intervals, considering the scaling used. As an upper
bound $m$ for the number of contained grid points we get
$$
m \le \left\lfloor \frac 1 2 \frac{\sqrt{3x'+3y'}}{4}\right\rfloor +1
\le \frac{\sqrt{6\rho}}{8}(n^{1/4}+O(n^{1/8}))+1
\le \frac 1 2 \sqrt{\rho}n^{1/4} +O(n^{1/8}).
$$

Combining our observations and estimates we get to bound by the averages as
\begin{equation*}
  \beta_3 \ge \frac{\min(q_-,q_+)}{m} \ge \frac{\Phi([3/2, 2]) + O(n^{-1/4})}{(1/2) \sqrt{\rho}n^{1/4} + O(n^{1/8})} =
  \frac{\beta_4}{\sqrt{\rho}n^{1/4}} + O(n^{-1/8}).
\end{equation*}
Finally, plugging this bound on $\beta_3$ into \eqref{eq:binom-global}
we arrive at
\begin{equation*}
  \mP(E_r) \ge \frac{\gamma_1}{\sqrt{\rho}n^{1/4}},
\end{equation*}
for any $0<\gamma_1<\beta_4/4$ and $n$ large enough, which matches the claim of the lemma.

\end{proof}

\section{Mapping the grid to the cycle}
\label{sec:gridmap}

In the previous section using the grid we have abstractly identified
points $R$ at appropriate distance from the starting position and also
points $R_0\subseteq R$ which are reached with non-negligible
probabilities. As the next step, we want to understand what these points represent on the original
cycle. W.l.g.\ we assume $k\le n/2$.

\begin{lemma}
  We have $\sqrt{\rho}n^{1/4}+O(1) \le |R_0| \le 4\sqrt{\rho}n^{1/4}+O(1)$.
\end{lemma}
\begin{proof}
For the moment, let us use the notation $\tilde{L}=L/n=\rho n^{1/2}+O(1)$.
Using this for both $x$ and $y$ would solve the defining equation
$kx+(n-k)y=L$.
Thus for any integer $x \in [\tilde{L}-\frac{1}{2}\sqrt{\rho}n^{1/4}+1,
\tilde{L}+\frac{1}{2}\sqrt{\rho}n^{1/4}-1]$ there is a corresponding
$y$ in the same interval that solves the defining equation. Here we
use the assumption $k \le n-k$, so that $y$ will differ no more from
the center value than $x$. Therefore $|x-y|\le \sqrt{\rho}n^{1/4}-2$
which is enough to ensure the pair accompanied by a $B$ being in $R_0$. The number of
integers in the given interval confirms the lower bound.

Adjusting the above argument, check the approximately double width interval $x \in [\tilde{L} - \sqrt{\rho}n^{1/4}, \tilde{L} +
  \sqrt{\rho}n^{1/4}], x \in \mZ$, which is a necessary condition
for being in $R_0$ and of form $(\cdot,\cdot,B)$. There will be at most one such point in
$R_0$ for each $x$, so $2\sqrt{\rho}n^{1/4}+O(1)$ in total. Now counting the
points on horizontal grid lines (collecting $(\cdot,\cdot,A)$ points), for any $y \in [\tilde{L} - \sqrt{\rho}n^{1/4}, \tilde{L} +
  \sqrt{\rho}n^{1/4}], y \in \mZ$ there will be at most one matching
$x\in\mR$ again. Adding up the two cases we get the upper bound.
\end{proof}

It will be easier to handle a set of points of known size, so let
$R_1\subseteq R_0$ be a subset of size $\sqrt{\rho}n^{1/4}$ (or maybe
$O(1)$ less) of the elements in the middle.

We want to convert our grid representation to the cycle and acquire
the image of $R_1$, that is,
$$
V_1 = \{g(r) ~|~ r\in R_1\}.
$$
To understand this set we scan through the elements of $R_1$, starting
with the one with the
lowest $x$ and increasing (taking direction $A$ before $B$ when passing through
a grid point) and follow where do they map on the cycle. When moving from
one point, $U$, to the next, $U'$, we may encounter the following
configurations, as shown in Figure \ref{fig:uustep}:

\renewcommand{\thesubfigure}{\roman{subfigure}}
\begin{figure}[h]
  \centering
  \hspace{0.03\textwidth}
  \subfloat[ ]{
    \includegraphics[width=0.15\textwidth]{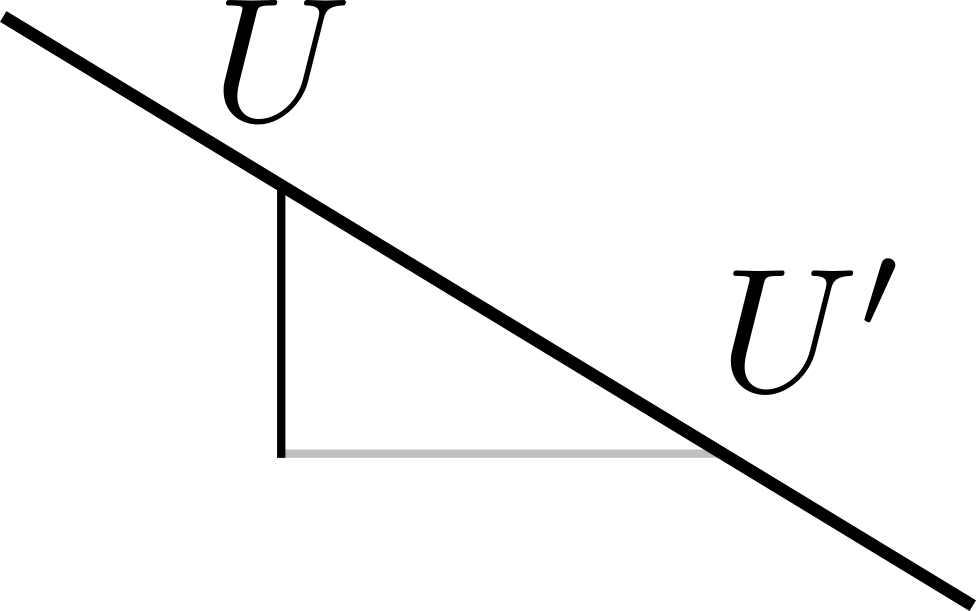}
  }
  \hspace{0.03\textwidth}
  \subfloat[ ]{
    \includegraphics[width=0.15\textwidth]{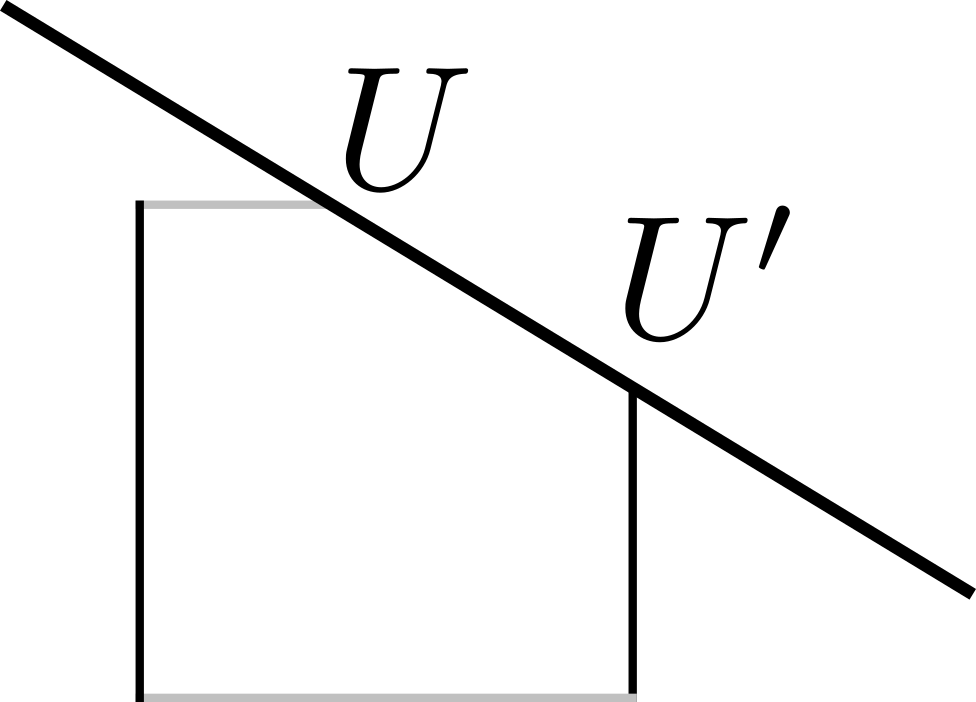}
  }
  \hspace{0.03\textwidth}
  \subfloat[ ]{
    \includegraphics[width=0.15\textwidth]{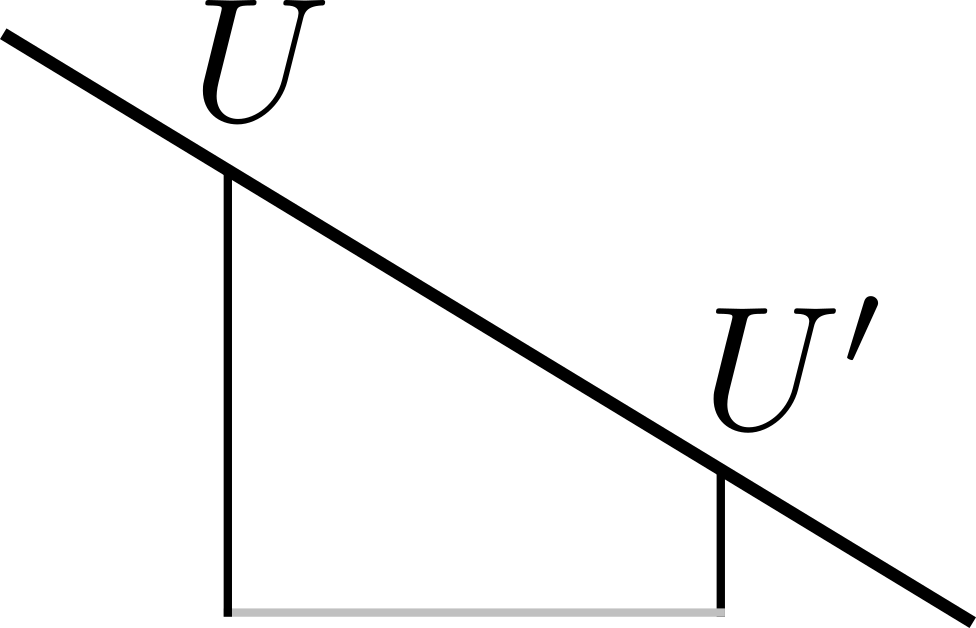}
  }
  \hspace{0.03\textwidth}
  \subfloat[ ]{
    \includegraphics[width=0.15\textwidth]{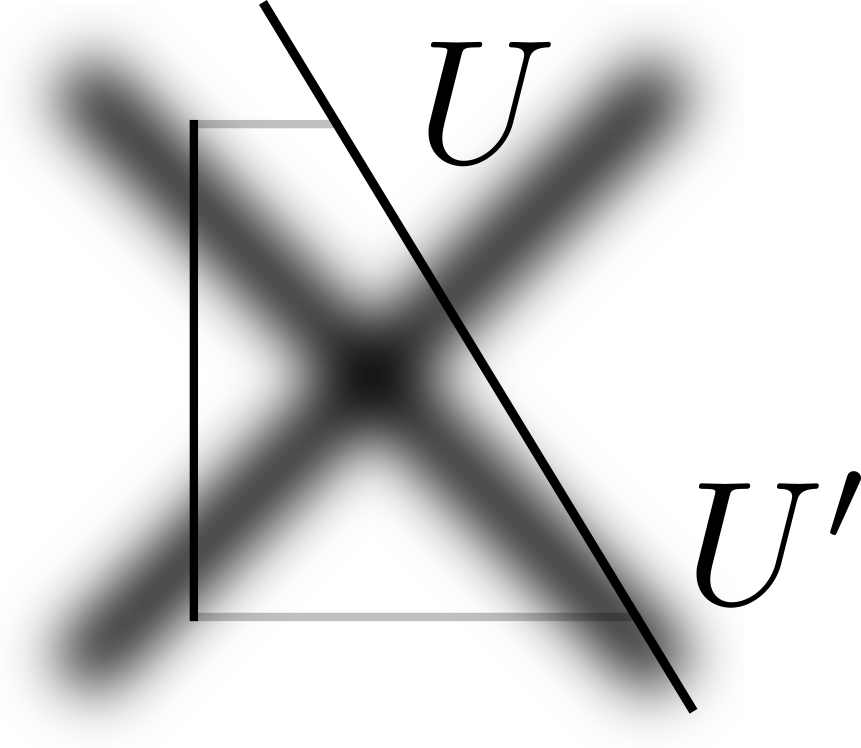}
  }
  \hspace{0.03\textwidth}
  \caption{Possible transitions depicted as $U\rightarrow U'$ while
    scanning through $R_1$.}
  \label{fig:uustep}
\end{figure}
\bi
\item In case (i) we see that the final few $l$ steps (out of $L$) are
  starting on the $B$ arc for $U$ and starting on the $A$ arc for
  $U'$. Consequently, $U'$ can be reached from $U$ by exactly $k$
  counter-clockwise steps on the cycle.
  
\item In case (ii) we almost reach the next grid point, some $l$ steps are
  missing on the $A$ arc for $U$ and also $l$ steps are missing on the $B$
  arc for $U'$. This means, again, that $U'$ can be reached from $U$ by exactly $k$
  counter-clockwise steps on the cycle.
  
\item In case (iii) we are on the $B$ arc for both $U$ and $U'$ but we had one
  more horizontal $A$ segment for $U'$ (representing $k$ steps) which is
  missing from the height. Therefore we get the same, $U'$ can be reached from $U$ by exactly $k$
  counter-clockwise steps on the cycle.
  
\item Passing through a grid point can be treated as the special case of
  either (i) or (ii), with the same consequence. Note that case (iv) can not happen due to our assumption of $k\le
  n/2$. 
\ei
To sum up, we can generate the set $V_1\subset V$ on the
cycle corresponding to $R_1$ by finding the first vertex, then taking
jumps of $-k$ (modulo $n$) for $|R_1|-1$ more steps. We want to ensure
the elements of $V_1$ are spread out enough. 
\begin{lemma}
  \label{lm:biggaps}
  There exist constants $\gamma_2 > 0$ and $\gamma_3 > \frac 1 2$ such that
  the following holds for large enough $n$. For a uniform choice of $k\in [2,3,\ldots,n-2]$ with
  probability at least $\gamma_2$ we have
  \begin{equation}
    \label{eq:farv}
    \forall v,v'\in V_1,~ v\neq v':~|v-v'|\geq \frac{\gamma_3}{\sqrt{\rho}} n^{3/4}.
  \end{equation}
\end{lemma}
\begin{proof}
  We will use $\gamma_3$ as a parameter for now which we will specify
  later. We will consider a uniform choice $k\in [1,2,\ldots,n]$ for
  convenient calculations, clearly this does not change the
  probability asymptotically.

  Two elements of $V_1$ get close if after some repetitions of
  the $k$-jumps, we get very close to the start (after a few full turns). More precisely the
  condition is violated iff
  $$
  \exists 1\le m \le |R_1|-1 \quad |mk| < \frac{\gamma_3}{\sqrt{\rho}} n^{3/4}.
  $$
  Our goal is to have a $k$ so that this not happen. For a fixed $m$
  this excludes $k$ from the intervals
  $$
  \left(\frac{in - \frac{\gamma_3}{\sqrt{\rho}} n^{3/4}}{m},\frac{in +
      \frac{\gamma_3}{\sqrt{\rho}} n^{3/4}}{m}\right),\quad i=0,1,\ldots,m-1.
  $$
  To simplify our calculations we will treat these intervals as real
  intervals on the cycle (rather than an interval of integers). Length
  and number of integers contained differ by at most one, we will
  correct for this error at the end.

  We need to merge these intervals for all $1\le m \le |R_1|-1$. We
  imagine doing this by collecting the intervals as increasing
  $m$. Observe that if $\gcd(i,m)=c>1$, then we already covered the
  interval around $\frac{in}{m}$ when encountering
  $\frac{(i/c)n}{m/c}$, and by a wider interval. That is, we only
  have to count those $i$ where $\gcd(i,m)=1$. Therefore the
  total newly covered area at step $m$ is at most
  $$
  \frac{2\gamma_3}{\sqrt{\rho}} n^{3/4} \frac{\varphi(m)}{m},
  $$
  where $\varphi$ denotes the classical Euler function. Once we add
  these up, and use the summation approximation
  \cite{walfisz1963weylsche}
  we get
  \begin{equation}
    \label{eq:sum_interval}
    \frac{2\gamma_3}{\sqrt{\rho}} n^{3/4}\sum_{m=1}^{|R_1|-1}\frac{\varphi(m)}{m} = 2\gamma_3\frac{6}{\pi^2}n
    + O(n^{3/4}\log n),
  \end{equation}
  knowing that $|R_1| = \sqrt{\rho}n^{1/4}+O(1)$. When we switched from integer
  counts to approximation by interval lengths, the total error is at
  most 1 per interval, that is,
  \begin{equation*}
    \sum_{m=1}^{|R_1|-1} m = O(n^{1/2}),
  \end{equation*}
  which is negligible compared to the quantities of \eqref{eq:sum_interval}.
  Consequently \eqref{eq:sum_interval} is an upper bound on the number of $k$ that should be
  excluded. Let us therefore choose $\gamma_3>1/2$ so that the
  coefficient of $n$ above is strictly less than 1. Then there is
  still strictly positive probability to pick a good $k$, in
  particular
  $$
  \gamma_2 = 1 - \frac{12\gamma_3}{\pi^2}-\varepsilon
  $$
  is adequate for any small $\varepsilon > 0$.
\end{proof}

\section{Including diffusive behavior}
\label{sec:diffusion}

So far we have understood the position of the chain after $L$ moves
from the first grid point.
Now we want to analyze the true Markov chain dynamics where moving or
staying in place is also random.
In $V_1$ we have a large number of positions, hopefully
different and separated enough, and we can bound the probability of reaching
the corresponding elements in $R_1$.

For technical reasons to follow, we want to avoid points in $V_1$ that
are very close to the hubs so define
\begin{align}
  V_2 &= \{v\in V_1 ~|~ |v-k|,|v-n| > 4\sqrt{\rho}n^{3/4}\sqrt{\log n}\},\nonumber\\
  W &= \{ w \in V ~|~ \exists v\in V_2,~|w-v| < \gamma_3\sqrt{\rho} n^{3/4}/2\},\label{eq:wdef}\\
  R_2 &= \{r\in R_1 ~|~ g(r) \in V_2\}.\nonumber
\end{align}
We would like to emphasize that in the favorable case when
\eqref{eq:farv} holds (ensured with positive probability by Lemma
\ref{lm:biggaps}), we have
$|V_2|, |R_2| = \sqrt{\rho}n^{1/4} + O(\sqrt{\log n})$ and when
$\rho\le 1$ this also implies
$|W| = \gamma_3 \rho n + O(n^{3/4}\sqrt{\log n})$.

At most vertices, a $Geo(1/2)$ distribution controls when to
step ahead, so let us choose some
$$T = 2\rho n^{3/2} + O(n) \in 2\mZ$$
and analyze
$X(T)$. Oversimplifying the situation at first, in $T$ steps the
chain travels $T/2$ in expectation, $O(n)$ to reach the origin grid point,
$\rho n^{3/2}+O(n)$ afterwards, which is exactly the case analyzed. We have
control of the probability of $\approx \sqrt{\rho} n^{1/4}$ of the
expected endpoints, and we will have a diffusion $\approx \sqrt{\rho}n^{3/4}$ around them, which together will provide a nicely spread out distribution.

However, there are some non-trivial details hidden here. The most important caveat is that when visiting the hubs, the
distribution of the time spent is not independent of the direction
taken. In fact, when arriving at a hub, say at vertex $a_k$, with probability $1/2$ there
is a single step, going to $b_1$. Otherwise, some loops, jumps are taken
between the hubs before moving on, which tells that with probability $1/4,1/4$ the
chain continues to $a_1,b_1$, in $1+Geo(1/2)$ steps. Let us combine
all the heuristics and work out the details in a formally precise way.

We are going to describe a procedure to generate the Markov chain and the position
$X(T)$. If $X(0)$ is $\lambda$ steps before one of the hubs, thus the origin grid point, fix
$L=T/2-\lambda = \rho n^{3/2}+O(n)$, define $R$ with this value in \eqref{eq:Rdef}.
Assume we are given an infinite i.i.d.\ series of fair
coin tosses which may either tell go ($1$) or do nothing ($0$).
We perform the following steps.
\begin{itemize}
\item Choose the exit point $r\in R$, with appropriate
  probability $\mP(E_r)$. 
\item Choose one of the possible tracks $\xi$ reaching $r$ (with the
  appropriate conditional probability).
\item Generate a series of coin tosses $c_0$ of length $T-(x'+y'+1)$,
  which is the major part of the movement of the chain.
\item Complement the above series depending on the
  track. Following the Markov chain using the beginning of $c_0$, when we reach a hub
  where the direction should be continued according to $\xi$ ($AA$ or $BB$), insert an extra
  $0$ symbol (correcting for the $1+Geo(1/2)$ waiting time
  distribution there). Similarly, when we reach a hub where the
  direction changes ($AB$ or $BA$), with probability $2/3$ insert a
  $1$ (meaning an instant step), with probability $1/3$ insert a $0$
  (for the $1+Geo(1/2)$ case). If we encounter a grid point further than
  $r$, we freely choose the direction together with the inserted symbol with
  probabilities $1/4,1/2,1/4$ for the 3 cases we had. Let the elongated series be $c_1$, the
  sequence of only the added symbols be $c_h$.
\end{itemize}
Let us use the notation $|c_0|$ for the length of $c_0$ and $\sum(c_0)$
for the number of $1$-s in $c_0$ (and similarly for the other
sequences). Let us also introduce $\tau = |c_1|$. Therefore at the
end of the procedure above we arrive at $X(\tau)$. More importantly, $\tau$
matches $T$ very well as stated below.
\begin{lemma}
  \label{lm:tauT}
  For $r\in R_2$ we have
  $$
  \mP(\tau = T) = 1 - O(n^{-4}).
  $$
\end{lemma}
\begin{proof}
  In $c_0$ for the number of $1$ symbols we have
  $$
  \left|\sum(c_0)-(L+\lambda)\right| \le \left|\sum(c_0)-\frac{1}{2}(T-(x'+y'+1))\right| + \left|\frac{1}{2}(T-(x'+y'+1))-(L+\lambda)\right|.
  $$
  The second term on the right hand side is $O(n^{1/2})$ by
  the definitions. For the term before we can use standard tail probability
  estimates for the Binomial distribution (based on Hoeffding's
  inequality). Merging the two error terms we get
  \begin{equation}
    \label{eq:sumc0_conc}
    \mP\left(\left|\sum(c_0)-(L+\lambda)\right| > 3\sqrt{\rho}n^{3/4}\sqrt{\log n}\right) = O(n^{-4}).
  \end{equation}
  Let us denote this bad event of the left side above by $\mathcal{B}$ for future use.
  Assume that this event does not occur, and $\sum(c_0)$ is within the
  error bound $3\sqrt{\rho}n^{3/4}\sqrt{\log n}$. This means that the Markov
  chain takes the first $\lambda$ steps to the origin and then
  $\approx L$ steps within the stated bounds.
  
  By the definition of $R_2$ this concentration tells that even only
  considering the $c_0$ steps we reach the grid line segment where $r$
  lies.
  On the way we pass through $x'+y'+1$ hubs, which results in
  $x'+y'+1=O(n^{1/2})$ entries in $c_h$. Conversely, inserting this
  $O(n^{1/2})$ steps into $c_0$ the upper bound ensures that we will
  not reach the next grid point (or the hub once more, in other words). Consequently, $|c_h|=x'+y'+1$.

  Therefore in this case we have $\tau = |c_1| = |c_0| + |c_h| = T$,
  which we wanted to show, the exceptional probability is controlled
  by \eqref{eq:sumc0_conc} which matches the claim.
\end{proof}
  
\begin{lemma}
  \label{lm:w_prob_bound}
  For any
  $w\in W$ we have
  \begin{equation*}
    \mP(X(T)=w) \ge \frac{\gamma_4}{\rho n},
  \end{equation*}
  for an appropriate global constant $\gamma_4>0$.
\end{lemma}
\begin{proof}
By the definition of $W$ there is a
$v = g(r)\in V_2$ with $|v-w| \le \frac{\gamma_3\sqrt{\rho}}{2}
n^{3/4}$ (if multiple, we choose one). We use the
procedure above to actually bound the probability for $X(\tau)$, but
by Lemma \ref{lm:tauT} we know this is correct up to $O(n^{-4})$ error
which is enough for our case. In the
process let us consider the case where $E_r$ is chosen and also
some track $\xi$ is fixed.

With these conditions, let us analyze the dependence structure of the
step sequences.
For $c_1$, the positions of the additions strongly depend on
$c_0$. However, we know exactly what hubs and which turns we
are going to take, which means $c_h$ is independent of $c_0$ (assuming
$\bar{\mathcal{B}}$), only their interlacing depends on both.

Now, first drawing and fixing $c_h$ we know by $\sum(c_h)$ precisely how many $1$-s do we need from
$c_0$ to exactly hit $w$ (still conditioning on $E_r$ and $\xi$). Let
this number be $s$, for which we clearly
have $|s-\rho n^{3/2}| \le \gamma_3\sqrt{\rho} n^{3/4}/2 + O(n^{1/2})$. The length of
$c_0$ is $T'=T-(x'+y'+1) = T + O(\sqrt{n})$. We have to approximate this
Binomial probability of $1$-s in $c_0$. This is a tedious calculation based on
the Stirling formula, we refer to \cite{spencer2014asymptopia} where it is shown that
$$
{T' \choose s} = (1+o(1))\frac{2^{T'}}{\sqrt{T'\pi/2}} \exp\left(-\frac{(T'-2s)^2}{2T'}\right),
$$
if $|T'/2-s| = o(T'^{2/3})$ which clearly holds in our case.
Substituting the variables we get the bound
$$
{T' \choose s} \frac{1}{2^{T'}} \ge (1+o(1))
\frac{1}{\sqrt{\pi\rho}n^{3/4}}
\exp\left(-\frac{\gamma_3^2}{4}+O(n^{-1/4})\right) \ge \frac{\beta_1}{\sqrt{\rho}n^{3/4}},
$$
for some constant $\beta_1>0$ and $n$ large enough. Thus, for
the conditional probability of interest we get
$$
\mP(X(T)=w ~|~E_r,\xi,c_h) \ge \frac{\beta_1}{\sqrt{\rho}n^{3/4}} -
\mP(\mathcal{B})
\ge \frac{\beta_2}{\sqrt{\rho}n^{3/4}}
$$
for any constant $\beta_1 > \beta_2 >0$ and $n$ large enough.

Observe that we have the same lower bound for $\mP(X(T)=w ~|~E_r)$
as it is a mixture of the conditional probabilities above, so we can
average out through $\xi$ and $c_h$. Finally, combining with
\eqref{eq:PEr-bound} we arrive at
\begin{equation*}
  \mP(X(T)=w) \geq \mP(X(T)=w ~|~E_r)\mP(E_r) \ge
  \frac{\gamma_4}{\rho n},
\end{equation*}
with an appropriate constant $\gamma_4 >0$.
\end{proof}

\section{Global mixing}
\label{sec:mixproof}

We now turn to evaluating the mixing metrics of our Markov chain. In
order to establish the upper bound on the mixing time initially
claimed in Theorem \ref{thm:main_upper_bound} we fix $\rho=1$ and
$T=2\lceil\rho n^{3/2}\rceil$ for this
section and use previous results using these parameters.
We will drop $X$ from indices and arguments in \eqref{eq:ddef},
\eqref{eq:tmixdef} when clear from the context.

An alternative of $d(t)$ compares the distribution of the Markov chain
when launched from two different starting points:
$$
\bar{d}(t) := \sup_{X^1(0),X^2(0)\in V} \|\cL(X^1(t))-\cL(X^2(t))\|_{\rm TV}.
$$
It is known how this compares with $d(t)$, we have the inequalities
$d(t) \le \bar{d}(t) \le 2d(t)$, moreover this variant is submultiplicative, $\bar{d}(s+t)\le\bar{d}(s)\bar{d}(t)$, see
\cite[Chapter~4]{levin2017markov}. We can quantify this distance for
our problem as follows.
\begin{lemma}
  \label{lm:dbar_bound}
  Assume that $n$ is large enough and $k$ is such that \eqref{eq:farv} holds.
  Then we have
  \begin{equation}
    \label{eq:dbar_bound}
    \bar{d}(T) \le 1 - \gamma_5,
  \end{equation}
  for some global constant $\gamma_5>0$.
\end{lemma}
\begin{proof}
  Fix two arbitrary starting vertices for $X^1(0)$ and $X^2(0)$ and
  denote the distribution of the two chains at time $T$ by
  $\sigma^1,\sigma^2$. Simple rearrangements yield
  \begin{equation}
    \label{eq:sigma12_TV_1}
    \begin{aligned}
    \|\sigma^1 - \sigma^2\|_{\rm TV} &= \frac 1 2 \sum_{v\in V}
    |\sigma^1(v) - \sigma^2(v)| = \frac 1 2 \sum_{v\in V} (\sigma^1(v) +
    \sigma^2(v) - 2\min(\sigma^1(v), \sigma^2(v)))\\
    &= 1 - \sum_{v\in V} \min(\sigma^1(v), \sigma^2(v)).
    \end{aligned}
  \end{equation}

  For both realizations in \eqref{eq:wdef} we get a subset of
  vertices, $W^1,W^2$, and there must be a
  considerable overlap, in particular
  $$
  |W^1\cap W^2| \ge |W^1|+|W^2|-n = (2\gamma_3 -1)n +
  O(n^{3/4}\sqrt{\log n}) \ge \beta n
  $$
  for some $\beta > 0$ and $n$ large enough, relying on the fact that $\gamma_3 > 1/2$.
  By Lemma \ref{lm:w_prob_bound} for any $w \in W^1\cap W^2$ we have
  both $\sigma^1(w),\sigma^2(w) \geq \gamma_4/n$. Substituting this
  back to \eqref{eq:sigma12_TV_1} we get
  \begin{equation*}
    \|\sigma^1 - \sigma^2\|_{\rm TV} \leq 1 - \beta n \frac{\gamma_4}{n} = 1 - \gamma_5,
  \end{equation*}
  with $\gamma_5 = \beta\gamma_4$. This upper bound applies for
  any two starting vertices of $X^1,X^2$, therefore the claim follows.
\end{proof}
We just need the final touch to prove Theorem \ref{thm:main_upper_bound}.
\begin{proof}[Proof of Theorem \ref{thm:main_upper_bound}]
Using Lemma \ref{lm:biggaps} we have a spread out collection in $V_2$
as stated in \eqref{eq:farv} with probability at least
$\gamma_2=:\gamma$. In this case we can apply Lemma \ref{lm:dbar_bound}. Fix
$$
T^* = \left\lceil \frac{\log \eps}{\log (1-\gamma_5)}\right\rceil T.
$$
Substituting \eqref{eq:dbar_bound} and using the basic properties of
$d,\bar{d}$ we get
$$
d(T^*)\le \bar{d}(T^*) \le \bar{d}(T)^{\left\lceil \frac{\log
      \eps}{\log (1-\gamma_5)}\right\rceil}
\le (1-\gamma_5)^{\left\lceil \frac{\log
   \eps}{\log (1-\gamma_5)}\right\rceil}
\le \eps.
$$
Consequently, $t_{\rm mix}(\eps) \le T^*$. On the other hand,
$$
T^* = \left\lceil \frac{\log \eps}{\log (1-\gamma_5)}\right\rceil T =
\left\lceil \frac{\log \eps}{\log (1-\gamma_5)}\right\rceil 2\lceil
n^{3/2}\rceil \le \gamma' n^{3/2} \log\frac{1}{\eps},
$$
for appropriate constant $\gamma'>0$ and $n$ large enough. Together with the
previous calculation this confirms the theorem.
\end{proof}

\section{Lower bound}
\label{sec:lowerbound}

In this section let us fix $\rho = \gamma_4/2$. Lemma
\ref{lm:w_prob_bound} tells that at elements of $W$ there is $2/n$
probability of the Markov chain to arrive, however, we have no information on the size
of $W$ for general $k$, for all edge selections, and there might be
significant overlaps, multiplicities when defining $V_1$ or $W$. It is
key for completing
the proof to be able to handle this.

  We know $|R_1|=\sqrt{\rho}n^{1/4} + O(1)$ but we have no size
  estimates on $|R_2|,|V_2|$. Let the set of interest be
  $$
  I = \{v\in V ~|~ |v-k|,|v-n| > 4\sqrt{\rho}n^{3/4}\sqrt{\log n}\}.
  $$
  For the moment we cannot guarantee that most of $R_1$ maps into $I$. In order to change this, we start moving time a bit.
  Let us define
  \begin{align*}
    T^i &= 2\lceil \rho n^{3/2} \rceil + 2i, \qquad i=0,1,\ldots,n-1,\\
    L^i &= \lceil \rho n^{3/2} \rceil - \lambda + i, \qquad
          i=0,1,\ldots,n-1.
  \end{align*}
  Accordingly, there is an evolution $R^i$ and $R_0^i$. Here $\lambda$
  corresponds to the position of $X(0)$ as before, which is arbitrary,
  but we consider it as fixed.

  For $R_1^i$ (together with $V_1^i$),
  we investigate this evolution from the perspective of the individual
  points. One can verify that a valid interpretation is that $r=(x,y,h)\in
  R_1^0$ performs a \emph{zigzag} away from the origin along the grid lines stopping just before
  $(x+1, y+1,h)$, turning at both grid points it passes along the way. During
  the move $|x'-y'|$ changes by at most 1, so there can be edge
  effects whether the point is allowed in $R_0^i$ or not, but only at the very ends
  of $R_1^i$. This process is illustrated in Figure \ref{fig:gridzigzag}.
  \begin{figure}[h]
    \centering
    \includegraphics[width=0.3\textwidth]{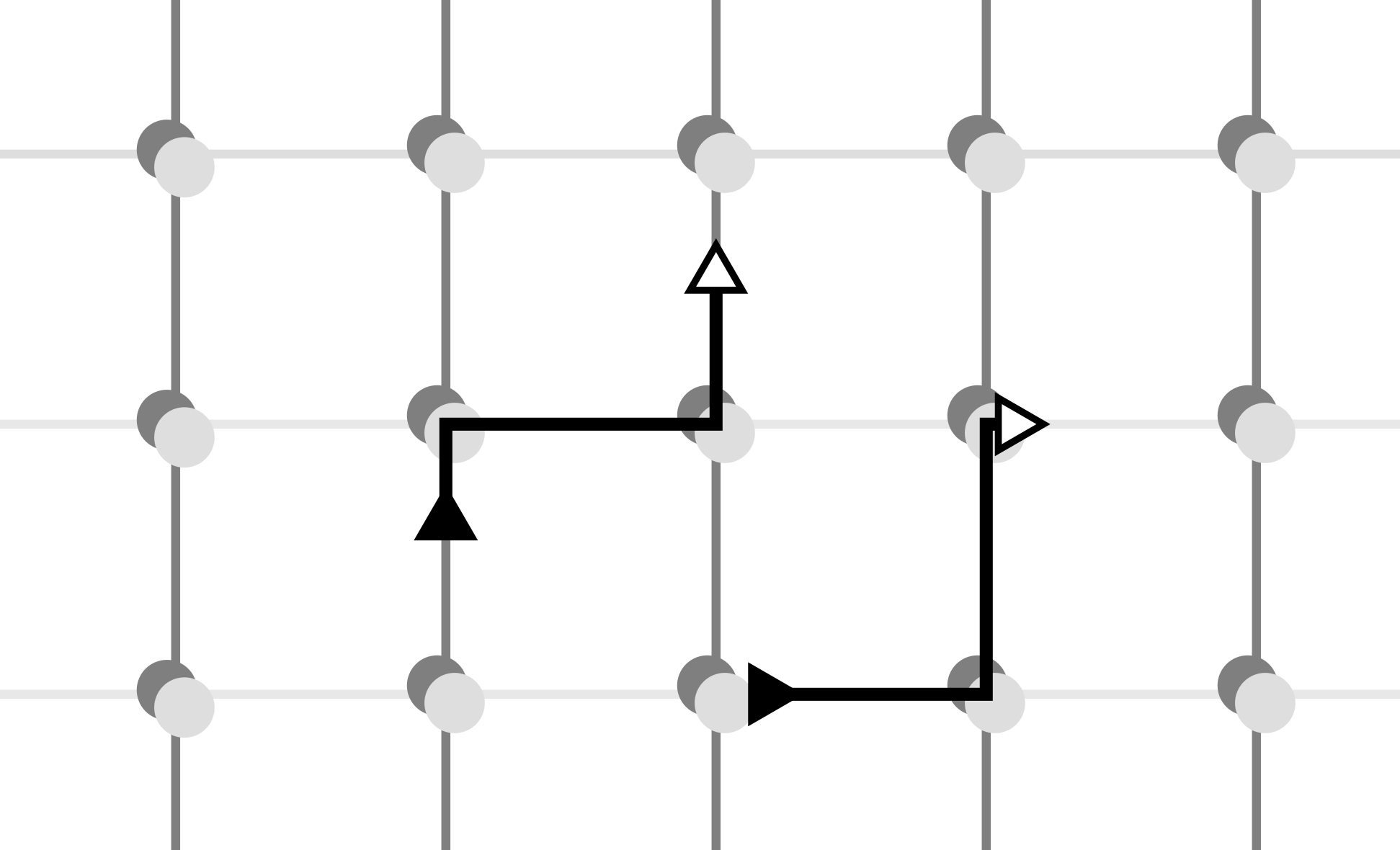}
    \caption{Evolution of the points of $R_1^i$ with the increase of
      $i$. Two such zigzag paths are depicted.}
    \label{fig:gridzigzag}
  \end{figure}
  Observe that such a zigzag corresponds to walking around the cycle
  exactly once. This will allow us to conveniently calculate the total number of hits to
  $I$ by $V_1^0,\ldots,V_1^{n-1}$. More precisely (to account for multiplicities), we are interested in
  $$
  M=\sum_{i=0}^{n-1} \sum_{r\in R_1^i} \m1_I (g(r)).
  $$
  By the previous argument we know that along each point the
  number of hits is
  exactly $I$,
  except maybe near the edges. Thus we get
  $$
  M \ge |R_1|\cdot |I| - O(n).
  $$
  Comparing the two while dividing by $n$ we get
  $$
  \frac{1}{n} \sum_{i=0}^{n-1} \sum_{r\in R_1^i} \m1_I (g(r))\ge |R_1|
  \frac{|I|}{n}-O(1)\ge \frac{|R_1|}{2}.
  $$
  Consequently, there has to be an index $i$ where the term on the left
  hand side being averaged out is larger than the right hand side:
  \begin{equation*}
  \sum_{r\in R_1^i} \m1_I (g(r)) \ge \frac{|R_1|}{2}.
  \end{equation*}
  Let us take and fix such an $i$ from now on. Define the
  analogous sets as before
  \begin{align*}
  R_2^i &= \{r\in R_1^i ~|~ g(r)\in I\},\\
  W^i &= \{w\in V ~|~ \exists r\in R_2^i,  |g(r)-w| \le \gamma_3\sqrt{\rho}n^{3/4}/2\}.
  \end{align*}
    
  We now formulate an extension of Lemma \ref{lm:w_prob_bound}:
  \begin{corollary}
    \label{cor:w_prob_bound}
    For any $r\in R_2^i$, $w\in W^i$ such that
    $|g(r)-w| \le \gamma_3\sqrt{\rho}n^{3/4}/2$ we have 
    \begin{equation*}
      \mP(X(T^i)=w, E_r) \ge \frac{\gamma_4}{\rho n},
    \end{equation*}
    for the same global constant $\gamma_4>0$ as before.
  \end{corollary}
  \begin{proof}
    This is actually what is happening under the hood in the proof of
    Lemma \ref{lm:w_prob_bound}, just there the choice of $r$ is given by
    the structure $W$ and the extra knowledge on $E_r$ is discarded at the end.
  \end{proof}

  \begin{proof}[Proof of Theorem \ref{thm:main_lower_bound}]
    We want to bound $d(T^i)$, with the $i$ carefully chosen
    above. 
    With only analyzing the case of the starting point fixed above
    we get a lower bound on
    $d(T^i)$. We need to estimate the total variation distance:
    $$
    \left\|\cL(X(T^i))-\frac{\m1}{n}\right\|_{\rm TV} = \frac 1 2 \sum_{v\in
      V}\left|P(X(T^i)=v)-\frac 1 n \right|.
    $$
    Note that for any $w\in W^i$ there is a corresponding $r$ nearby
    to apply Corollary \ref{cor:w_prob_bound}, and with the choice of
    $\rho$ at the beginning of the section we get $P(X(T)=w)\ge
    2/n$. That is, these terms in the sum are positive without the
    absolute value. We drop other values for a lower bound. Thus we get
    $$
    \left\|\cL(X(T^i))-\frac{\m1}{n}\right\|_{\rm TV} \ge \frac 1 2 \sum_{w\in
      W^i}\left(P(X(T^i)=w)-\frac 1 n\right)\ge \frac 1 2 \sum_{r,w} \frac 2 n -
    \frac 1 2 |W^i| \frac 1 n.
    $$
    In the first term we may include all compatible pairs of $r,w$ for
    which Corollary \ref{cor:w_prob_bound} can be applied. Recall
    that $|R^i_2| \ge \sqrt{\rho}n^{1/4}/2 + O(1)$, each element
    compatible with
    $\gamma_3 \sqrt{\rho}n^{3/4} + O(1)$ number of $w$'s.

    For the second term being subtracted, we may count very similarly
    starting from $R^i_2$ and looking for compatible pairs. This time
    however, multiplicity does not add up as we need the size of the set, but we want an upper bound
    for this term anyway. In total we get
    $$
    \left\|\cL(X(T^i))-\frac{\m1}{n}\right\|_{\rm TV} \ge
    \frac 1 n |R^i_2| \left({\gamma_3}{\sqrt{\rho}} n^{3/4} + O(1)\right) - 
    \frac 1 {2n} |R^i_2| \left({\gamma_3}{\sqrt{\rho}} n^{3/4} + O(1)\right) \ge
    \frac {\gamma_3\rho} 4 + O(n^{-1/4}).
    $$

    As $d(\cdot)$ is non-increasing,
    for the constants of the theorem we may choose $\gamma^*=2\rho=\gamma_4$
    and any $\eps^*<\gamma_3\rho/4 = \gamma_3\gamma_4/8$. With such choices the previous calculations show
    that the claim holds.
    
\end{proof}

\section{Discussion, conclusions}
\label{sec:discussion}

Let us further elaborate on the results obtained together with ideas for
possible extensions.

First of all, Theorem \ref{thm:main_upper_bound}
provides a global probability estimate for the mixing time bound to hold
but is not an a.a.s.\ result. It is unclear if this is the result of our
bounds being conservative, or because truly there is a large proportion
of badly behaving $k$. Note that there \emph{are} bad $k$, for
instance if $n$ is even then for $k=n/2$ the mixing time is truly
quadratic in $n$.

Concerning the proof technique, observe that the grid together with
the grid lines correspond to a covering of the original graph which is
well suited for our purpose.

An interesting question is whether there is an extension possible for
slightly more connections. The more natural one is to increase the
number of random edges. In this case however, one might need to handle
the effects of the small permutations generated by the various edges.
The more accessible one is to increase the
number of random hubs, then adding all-to-all connections between
them. Closely related work has been done for the asymptotic rate of
convergence \cite{gb_jh:rmplusclique2015} when then number of hubs can grow at any
small polynomial rate, and it turns out that in that case the inverse
spectral gap is linear in
the length of the arcs (excluding logarithmic factors), which would be a bottleneck anyway.

Still, we might guess the mixing time in a heuristic manner for constant number of hubs,
with the generous assumption that our concepts can be carried
over. Let us consider $K$ random hubs -- and thus also $K$ arcs -- and check mixing until
$n^\alpha$ for some $\alpha$. We can assume the lengths of the arcs
are of order $n$, so the set generalizing $R$ is on a $K-1$ dimensional
hyperplane at distance $\approx n^{\alpha-1}$ from the origin. We can hope to get a
CLT type control on the probability in a ball of radius $\approx
n^{(\alpha-1)/2}$ again, thus the number of such points in the hyperplane is $\approx
n^{(K-1)(\alpha-1)/2}$. If once again these map to far away points on
the cycle and this movement can be nicely blended together with the
diffusion, that would provide an extra factor of $\approx n^{\alpha/2}$, meaning
the total number of vertices reached is
$$
\approx n^{(K-1)\frac{\alpha-1}{2}+\frac{\alpha}{2}}
=n^{\frac{K}{2}(\alpha-1)+\frac 1 2}.
$$
We hope for mixing when the exponent reaches 1 which translates to
$\alpha = 1 + 1/K$, and leads us to the following conjecture:
\begin{conjecture}
  Consider $K\in \mZ_+, \eps>0$ fixed. On a cycle of $n$ vertices, choose $K$
  hubs randomly. With an appropriate interconnection structure among
  the hubs, and a Markov chain otherwise analogous to the one before
  there is a positive bounded probability to have
  $$
  t_{\rm mix}(\eps) = \Theta\left(n^{1+1/K}\right),
  $$
  for $n$ large enough.
\end{conjecture}

Clearly there is quite some flexibility left for the class of Markov
chains to consider. This remains as a question for future research to
find out exactly what is needed to generalize the current toolchain.

\bibliographystyle{siam}
\bibliography{mcmt,ringmixing,current}

\end{document}